\documentclass[a4paper,11pt]{article}
\usepackage{amssymb}
\usepackage{amsthm}
\usepackage{amsmath}
\usepackage{enumerate}
\usepackage{verbatim}
\usepackage{color}
\usepackage[left=2cm,right=2cm,top=2cm,bottom=2cm,includefoot]{geometry}
\title{}
\author{}
\newtheorem{theorem}{Theorem}
\newtheorem{lemma}[theorem]{Lemma}

\newtheorem{definition}{Definition}
\newtheorem{megaclaim}{Claim}
\newcommand{\claim}[2]{\begin{megaclaim}\label{#1} #2 \end{megaclaim}}
\newcommand{\refclaim}[1]{Claim~\ref{#1}}

\newcommand{\PP}{\mathcal{P}}
\newcommand{\zX}{\mathcal{X}}
\newcommand{\zY}{\mathcal{Y}}

\title{A structure theorem for strong immersions}
\author{Zden\v ek Dvo\v r{\'a}k\thanks{Computer Science Institute of Charles University, Prague, Czech Republic.
E-mail: \texttt{rakdver@iuuk.mff.cuni.cz}.  Supported the Center of Excellence -- Inst. for Theor. Comp. Sci., Prague (project P202/12/G061 of Czech Science Foundation), and
by project LH12095 (New combinatorial algorithms - decompositions, parameterization, efficient solutions) of Czech Ministry of Education.
}\and Paul Wollan\thanks{
Department of Computer Science, University of Rome, ``La Sapienza", Rome, Italy \texttt{wollan@di.uniroma1.it}. 
Supported by the European Research Council under the European Union's Seventh Framework Programme (FP7/2007-2013)/ERC Grant Agreement no. 279558.}
\\
}
\begin{document}

\maketitle

\begin{abstract}
A graph $H$ is \emph{strongly immersed} in $G$ if $H$ is obtained from $G$ by a sequence of vertex splittings (i.e., lifting some pairs
of incident edges and removing the vertex) and edge removals.
Equivalently, vertices of $H$ are mapped to distinct vertices of $G$ (\emph{branch vertices}) and edges of $H$ are mapped to pairwise edge-disjoint
paths in $G$, each of them joining the branch vertices corresponding to the ends of the edge and not containing any other branch vertices.
We describe the structure of graphs avoiding a fixed graph as a strong immersion.
\end{abstract}

\section{Introduction}\label{sec:intro}

In this paper, we consider graphs which can have parallel edges and loops, where each loop contributes $2$ to the degree of the incident vertex.  A graph without parallel edges and loops is called \emph{simple}.

Various containment relations between graphs have been studied in structural graph theory.  The best known ones are perhaps \emph{minors} and \emph{topological minors}.
A graph $H$ is a \emph{minor} of $G$ if it can be obtained from $G$ by a sequence of edge and vertex removals and edge contractions.
A graph $H$ is a \emph{topological minor} of $G$ if a subdivision of $H$ is a subgraph of $G$, or equivalently, if $H$ can be obtained from $G$
by a sequence of edge and vertex removals and by suppressions of vertices of degree two.  In their fundamental series of papers, Robertson and Seymour developed
the theory of graphs avoiding a fixed minor, giving a description of their structure~\cite{robertson2003graph} and proving that every proper minor-closed class
of graphs is characterized by a finite set of forbidden minors~\cite{rs20}.  The topological minor relation is somewhat harder to deal with (and in particular,
there exist proper topological minor-closed classes that are not characterized by a finite set of forbidden topological minors), but a description of their structure
is also available~\cite{gmarx,topmin}.

In this paper, we consider the related notion of a graph immersion. Let $H$ and $G$ be graphs.  
An \emph{immersion} of $H$ in $G$ is a function $\theta$ from vertices and edges of $H$ such that
\begin{itemize}
\item $\theta(v)$ is a vertex of $G$ for each $v\in V(H)$, and $\theta\restriction V(H)$ is injective.
\item $\theta(e)$ is a connected subgraph of $G$ for each $e\in E(H)$, and if $f\in E(H)$ is distinct from $e$, then
$\theta(e)$ and $\theta(f)$ are edge-disjoint.
\item If $e\in E(H)$ is incident with $v\in V(H)$, then $\theta(v)$ is a vertex of $\theta(e)$,
and if $e$ is a loop, then $\theta(e)$ contains a cycle passing through $\theta(v)$.
\end{itemize}
An immersion $\theta$ is \emph{strong} if it additionally satisfies the following condition:
\begin{itemize}
\item If $e\in E(H)$ is not incident with $v\in V(H)$, then $\theta(e)$ does not contain $\theta(v)$.
\end{itemize}
When we want to emphasize that an immersion does not have to be strong, we call it \emph{weak}.

If $H$ is a topological minor of $G$, then $H$ is also strongly immersed in $G$.  On the other hand,
an appearance of $H$ as a minor does not imply an immersion of $H$, and conversely, an appearance of $H$ as a strong immersion
does not imply the appearance as a minor or a topological minor.  Nevertheless, many of the results for minors and topological minors
have analogues for immersions and strong immersions.  For example, any simple graph with minimum degree at least $200k$ contains
a strong immersion of the complete graph $K_k$ (DeVos et al.~\cite{mindim}), as compared to similar sufficient minimum degree conditions
for minors ($\Omega(k\sqrt{\log k})$, Kostochka~\cite{kostomindeg}, Thomason~\cite{thommindeg}) and topological minors ($\Omega(k^2)$,
Bollob\'as and Thomason~\cite{BoTh}, Koml\'os and Szemer\'edi \cite{KoSz}).
Furthermore, every proper class of graphs closed on weak immersions is characterized by a finite set of forbidden immersions~\cite{rs23}.

Let us restrict our attention for the moment to weak immersions.  Fix $t$ to be a positive integer.  It is easy to show that if a graph $G$ contains a set $X$ of $t+1$ vertices such that for every pair of vertices $x, y \in X$ there does not exist an edge cut of order less than $t^2$ serparating $x$ and $y$, then $G$ contains $K_t$ as a weak immersion.  From this observation, we see that any graph which does not contain $K_t$ as a weak immersion must either have a small number of large degree vertices, or alternatively, there exists a small edge cut separating two big degree vertices.  This gives rise to an easy structure theorem for weak immersions as shown in DeVos et al.~\cite{mattpriv} and Wollan~\cite{wollims}.  

The same is not true for strong immersions.  There exist graphs which are arbitrarily highly edge connected and still have no strong immersion of $K_3$ (although such graphs will necessarily not be simple by the extremal result of \cite{mindim} mentioned above).  As an example, let $k$ be a positive integer and consider the graph $P_k$ obtained from from a path of length $k$ by adding $k-1$ parallel edges to each edge of the path.  Then $P_k$ is $k$ edge connected but does not contain even $K_3$ as a strong immersion.  The example can be made more complex as well.  If we add edges to connect every pair of vertices at distance two on the path, the resulting graph will still have no strong immersion of $K_4$.  When the graph is assumed to be highly edge connected, this is essentially the only obstruction to a graph excluding a strong immersion of a fixed clique as shown by Marx and Wollan \cite{MW}.

The main result of this article is a decomposition theorem for strong immersions.  The basis is a theorem which says that if a graph avoids a strong immersion of a fixed clique and contains a large set $X$ of vertices which are pairwise highly edge
connected, then the graph must have a decomposition with respect to $X$ yielding an obstruction to the existence of a strong clique immersion,
similar in spirit to the construction of the highly edge connected graph avoiding a strong clique immersion described in the previous paragraph.
From this decomposition theorem, it is straightforward to derive a structure theorem for graphs excluding a fixed graph as a strong immersion.  

We first rigorously define the decomposition which arises.  
A \emph{near-partition} of a set $Z$ is a family of subsets $Z_1, \dots, Z_k$, possibly empty, such that $\bigcup_1^k Z_i = Z$ and $Z_i \cap Z_j = \emptyset$ for all $1 \le i < j \le k$.  
Let $X$ be a subset of vertices of a graph $G$.  A \emph{path-like decomposition $\PP$ of $G$ with respect to $X$} is an ordering
$x_1$, $x_2$, \ldots, $x_t$ of the vertices of $X$ and a near-partition $B_0$, \ldots, $B_t$ of $V(G)\setminus X$.  The elements of the near-partition
are called \emph{bags} of the decomposition.  For a vertex $x_i\in X$,
the \emph{$x_i$-cut} of the decomposition is the set of edges of $G$ with one endpoint in $\{x_1,\ldots, x_{i-1}\}\cup B_0\cup \ldots \cup B_{i-1}$
and the other endpoint in $\{x_{i+1},\ldots, x_t\}\cup B_i\cup \ldots \cup B_t$.  The \emph{width} of the decomposition is the maximum size
of an $x$-cut with $x\in X$.  For a set $Z\subseteq V(G)$, let $b_{\PP}(Z)$ denote the number of bags of the decomposition that intersect $Z$.
For an integer $p\ge 0$, we say that $Z$ is \emph{$p$-bounded} in the decomposition $\PP$ when $|X\cap Z|+b_{\PP}(Z)\le p$.
For an integer $k\ge 1$, we say that a set $W\subseteq V(G)$ is \emph{$k$-edge-connected} if no two vertices of $W$ are separated by an edge-cut of size less than $k$ in $G$.
A set $W\subseteq V(G)$ is \emph{$\alpha$-linear} if there exists a set $A\subseteq W$ of size at most $\alpha$
such that $G-A$ has a path-like decomposition $\PP$ with respect to $W\setminus A$ of width less than $\alpha$
and the neigborhood of every vertex of $A$ is $\alpha$-bounded in $\PP$.  

The following theorem is the main technical result of this article.  It shows that every highly edge-connected set in a graph $G$ is $\alpha$-linear for some bounded value $\alpha$ if $G$ does not contain a strong immersion of a fixed graph $F$.
 
\begin{theorem}\label{thm-main}
For every graph $F$, there exists a value $\alpha$ such that if a graph $G$ does not contain $F$ as a strong immersion, then every $\alpha$-edge-connected subset of $V(G)$ is $\alpha$-linear.
\end{theorem}

Conversely, we show as well that the property of being $\alpha$-linear is a good approximate characterization of graphs excluding a fixed immersion in that every graph which satisfies this property for every highly edge connected set cannot contain a strong immersion of a big clique.

Assuming Theorem \ref{thm-main}, we can derive a global structure theorem for graphs excluding a fixed graph as a strong immersion.  The theorem and its proof are presented in Section \ref{sec:struct}.  The proof of Theorem \ref{thm-main} is given in Section \ref{sec:decomp} and in the final section, we show that the converse statement to the decomposition and structure theorem are approximately true.

We establish some notation we will use going forward.  Let $G$ be a graph and $X \subseteq V(G)$.  We use $G[X]$ to denote the subgraph of $G$ induced by $X$.  The graph $G-X$ refers to the subgraph of $G$ induced on $V(G)\setminus X$.  For a subset $K$ of edges, $G-K$ is the subgraph with vertex set $V(G)$ and edge set $E(G) \setminus K$.  We will use $\delta(X)$ to refer to the set of edges with one endpoint in $X$ and one endpoint not in $X$ (specifically, $\delta(X)$ does not contain any loops).  A \emph{separation} of $G$ is a pair $(X, Y)$ of non-empty subsets of $V(G)$ such that $X \cap Y = \emptyset$ and every edge has all it's endpoints either contained in $X$ or contained in $Y$.

\section{A structure theorem}\label{sec:struct}

In this section, we show how Theorem \ref{thm-main} gives rise to a global structure theorem for graphs which exclude a fixed graph $H$ as a strong immersion.  We first present some further notation.

When studying graph minors, a natural decomposition is to break the graph on a small vertex cutset and look at the structure on each side of the cutset.  This gives rise to the operation of clique sums on graphs.  Given that graph immersions consist of a set of edge disjoint paths, it is natural to instead look at when the graph can be decomposed on a small edge cut.  This motivates the definition of what we will refer to as edge sums in graphs. 

\begin{definition}
Let $G$, $G_1$, and $G_2$ be graphs.  Let $k \ge 1$ be a positive integer.  The graph $G$ is a \emph{$k$-edge sum} of $G_1$ and $G_2$ if the following holds.  There exist vertices $v_i \in V(G_i)$ such that $\deg(v_i) = k$ for $i = 1, 2$ and a bijection $\pi: \delta(v_1) \rightarrow \delta(v_2)$ such that $G$ is obtained from $(G_1 - v_1) \cup (G_2 - v_2)$ by adding an edge from $x \in V(G_1) - v_1$ to $y\in V(G_2) - v_2$ for every pair $e_1, e_2$ of edges such that $e_1 \in \delta(v_1)$, $e_2=\pi(e_1)$, the ends of $e_1$ are $x$ and $v_1$, and the ends of $e_2$ are $y$ and $v_2$.

We will also refer to a $k$-edge sum as an edge sum of \emph{order} $k$.  The edge sum is \emph{grounded} if  there exist vertices $v_1'$ and $v_2'$ in $G_1$ and $G_2$, respectively, such that for $i = 1, 2$, $v_i' \neq v_i$ and there exist $k$ edge-disjoint paths linking $v_i$ and $v_i'$.  If $G$ can be obtained by a $k$-edge sum of $G_1$ and $G_2$, we write $G = G_1 \hat{\oplus}_k G_2$.
\end{definition}

The following lemma appears in \cite{wollims} and shows that edge sums preserve the presence of immersions.
\begin{lemma}[\cite{wollims}]\label{lem:sum}
Let $G$, $G_1$, and $G_2$ be graphs and let $k \ge 1$ be a positive integer.  Assume $G = G_1 \hat{\oplus}_k G_2$, and assume that the edge sum is grounded.  Let $H$ be an arbitrary graph.  If $G_1$ or $G_2$ admits an immersion of $H$, then $G$ does as well.  If the immersion in either $G_1$ or $G_2$ is strong, then the immersion in $G$ is also strong.  
\end{lemma}

Just as clique sums give rise to tree decompositions, edge sums yield a natural tree-like decomposition of graphs.

\begin{definition}
A \emph{tree-cut decomposition} of a graph $G$ is a pair $(T, \zX)$ such that $T$ is a tree and $\zX = \{X_t \subseteq V(G): t \in V(T)\}$ is a near-partition of the vertices of $G$ indexed by the vertices of $T$.  For each edge $e = uv$ in $T$, $T-uv$ has exactly two components, namely $T_v$ and $T_u$ containing $v$ and $u$ respectively.  The \emph{adhesion} of the decomposition is 
\begin{equation*}
max_{uv \in E(T)} \left | \delta\left ( \bigcup_{t \in V(T_v)} X_t \right)\right |
\end{equation*}
when $T$ has at least one edge, and 0 otherwise.  
The sets $\{X_t: t \in V(T)\}$ are called the \emph{bags} of the decomposition.
\end{definition}
Note that the definition allows bags to be empty.  

We will need to define one more operation on graphs.  Let $G$ be a graph and $X \subseteq V(G)$.  The graph $G'$ is obtained by \emph{consolidating $X$} if we identify the vertices of $X$ to a single vertex $v_X$ and delete all loops incident to $v_X$.

Let $G$ be a graph and $(T, \zX)$ a tree-cut decomposition of $G$.  Fix a vertex $t \in V(T)$.  The \emph{torso of $(G, T, \zX)$ at $t$} is the graph $H$ defined as follows.  If $|V(T)| = 1$, then the torso $H$ of $(G, T, \zX)$ at $t$ is simply $G$ itself.  If $|V(T)| \ge 2$, let the components of $T-t$ be $T_1, \dots, T_l$ for some positive integer $l$.  Let $Z_i = \bigcup_{x \in V(T_i)} X_x$ for $1 \le i \le l$.  Then $H$ is made by consolidating each set $Z_i$ to a single vertex $z_i$.  The vertices $X_t$ are called the \emph{core vertices} of the torso.  The vertices $z_i$ are called the \emph{peripheral vertices} of the torso.  When there can be no confusion as to the graph $G$ in question, we will also refer to the torso of $(T, \zX)$ at a vertex $t$.

The following lemma shows that tree cut decompositions can be combined in an edge sum of graphs.
\begin{lemma}[\cite{wollims}]\label{lem:tredecom}
Let $G$, $G_1$, and $G_2$ be graphs such that $G = G_1 \hat{\oplus}_k G_2$ for some $k \ge 0$.  If $G_i$ has a tree-cut decomposition $(T_i, \zX_i)$ for $i = 1, 2$, then $G$ has a tree-cut decomposition $(T, \zY)$ such that the adhesion of $(T, \zY)$ is equal to 
\begin{equation*}
max \{ k, adhesion(T_1, \zX_1), adhesion(T_2, \zX_2)\}.
\end{equation*}
Moreover, for every $t \in V(T)$, there exists $i \in \{1,2\}$ and a vertex $t'$ in $V(T_i)$ such that the torso $H_t$ of $(G, T, \zY)$ at $t$ is isomorphic to the torso $H'$ of $(G_i, T_i, \zX_i)$ at $t'$.  Finally, every core vertex of $H_t$ is a core vertex of $H'$.
\end{lemma}

We can now state the structure theorem for graphs excluding a fixed clique immersion in terms of a tree-cut decomposition.  The proof will follow easily assuming Theorem \ref{thm-main}.  We say that a graph is \emph{$\alpha$-basic}
if the set of all its vertices of degree at least $\alpha$ is $\alpha$-linear, i.e., it has a path-like decomposition such that
all the vertices in its bags have degree less than $\alpha$. 

\begin{theorem}\label{thm:weakdecomp2}
For every graph $F$, there exists an integer $\alpha = \alpha(F)$ such that if a graph $G$ does not contain $F$ as a strong immersion, then there exists a tree-cut decomposition $(T, \zX)$ of $G$ of adhesion less than $\alpha$ such that each torso is $\alpha$-basic.
\end{theorem}

\begin{proof}[Proof (assuming Theorem \ref{thm-main})]
Fix the graph $F$ and let $\alpha$ be the value given in Theorem \ref{thm-main}.  Assume the statement is false, and let $G$ be a counterexample on a minimum number of edges.   The set of vertices of degree at least $\alpha$ in $G$ is not $\alpha$-edge-connected, as otherwise Theorem \ref{thm-main} yields a contradiction.  Thus, we may assume that there exists vertices $x$ and $y$ each of degree at least $\alpha$ such that there exists $X \subseteq V(G)$ with $x \in X$, $y \notin X$ and $|\delta(X)|  < \alpha$.  

Let $G_X$ be the graph obtained by consolidating $V(G) \setminus X$ and $G_Y$ the graph obtained by consolidating $X$.  By construction, $G = G_X \hat{\oplus}_k G_Y$ for some positive integer $k < \alpha$, and if we assume we chose a minimum order edge cut separating $x$ and $y$, the edge sum is grounded.  Given that both $x$ and $y$ have degree at least $\alpha$, we see that $|E(G_X)| < |E(G)|$ and $|E(G_Y)| < |E(G)|$.  By Lemma \ref{lem:sum}, neither $G_X$ nor $G_Y$ contains $F$ as a strong immersion.  Both $G_X$ and $G_Y$ have the desired decomposition by minimality, and therefore, by Lemma \ref{lem:tredecom}, $G$ has the desired decomposition as well.
\end{proof}

\section{Proof of Theorem \ref{thm-main}}\label{sec:decomp}

We prove a slightly stronger statement which gives a clearer picture on the relationship between the parameters involved.  Let $G$ be a graph and $a, w, p \ge 1$ be positive integers.  A set $W\subseteq V(G)$ is \emph{$(a,w,p)$-linear} if there exists a set $A\subseteq W$ of size at most $a$
such that $G-A$ has a path-like decomposition $\PP$ with respect to $W\setminus A$ of width less than $w$
and the neigborhood of every vertex of $A$ is $p$-bounded in $\PP$.

\begin{theorem}\label{thm:mainstrong}
For every graph $F$, there exist integers $a$, $w$ and $p$ such that
if a graph $G$ does not contain $F$ as a strong immersion, then every $w$-edge-connected subset of $V(G)$ is $(a,w,p)$-linear.
\end{theorem}

To see that Theorem \ref{thm:mainstrong} is in fact a strengthening of Theorem \ref{thm-main}, we observe the following.  Assume that a subset $W$ of vertices in a graph $G$ is $(a,w,p)$-linear for positive integers $a, w, p$.  Then a path-like decomposition of $G$ with respect to $W$ which certifies that $W$ is $(a, w, p)$-linear trivially certifies as well that $W$ is $(a', w', p')$-linear for all $a' \ge a$, $w' \ge w$, and $p' \ge p$.  Thus, $W$ is $\alpha$-linear for $\alpha = \max\{a, w, p\}$, implying that Theorem \ref{thm-main} is an immediate consequence of Theorem \ref{thm:mainstrong}.

For the remainder of this section, we define 
\begin{equation*}
d(k)=(2k+1)^{8k+4}k^2(k+1).
\end{equation*}
We use the following result of Dvo\v{r}\'ak and Klimo\v{s}ov\'a~\cite{dvoklim}.

\begin{theorem}[\cite{dvoklim}]\label{thm-sep}
Let $G$ be a graph and $x \in V(G)$.  Let $k\ge 3$ be an integer.
Let $Y\subseteq V(G)\setminus \{x\}$ be a set of vertices such that $G$ contains no edge cut of size less than $k$ separating $x$ from a vertex in $Y$.
If a graph $F$ of maximum degree at most $k$ does not appear in a graph $G$ as a strong immersion, then
there exist sets $Y'\subseteq Y$ and $K\subseteq E(G)$ such that $k|Y'|+|K|<d(k)|V(F)|$ and the component of $(G-Y')-K$ that
contains $x$ does not contain any vertex of $Y$.
\end{theorem}

For a graph $G$, a set $W\subseteq V(G)$ and an integer $m$, let $G(m,W)$ denote the graph with vertex set $W$ such that
two vertices $x$ and $y$ in $W$ are adjacent in $G(m, W)$ iff $G-(W\setminus \{x,y\})$ contains at least $m$ pairwise edge-disjoint paths
joining $x$ with $y$.  As a corollary of Theorem~\ref{thm-sep}, Dvo\v{r}\'ak and Klimo\v{s}ov\'a~\cite{dvoklim} proved
that if $W$ is sufficiently edge-connected and sufficiently large and $G$ avoids some fixed graph as a strong immersion, then $G(m,W)$ is connected.
We need a strenghtening of this claim.

\begin{lemma}\label{lemma-vcon}
Let $F$ be a graph and let $k=\max\{\Delta(F),3\}$.
For all integers $a_0,m\ge 0$, there exists $w\ge k$ such
that the following holds.  Let $G$ be a graph, let $W\subseteq V(G)$ be a $w$-edge-connected set
and let $A$ be a subset of $W$ of size at most $a_0$.
Suppose that $G(m,W)-A$ is not connected and let $(X_1, X_2)$ be a separation of $G(m, W)-A$.  If $G$ does not contain $F$ as a strong immersion, then $G-A$ contains an edge-cut of size
less than $w$ separating $X_1$ from $X_2$.
\end{lemma}
\begin{proof}
Let $s=d(k)|V(F)|$, $w_0=ms^3+s^2$ and $w=\max\{2a_0sw_0,w_0+a_0s\}$.

\claim{cl-single}{For $i\in\{1,2\}$, every vertex $x\in X_i$ is separated from $X_{3-i}$ by an edge-cut of size less than $w_0$ in $G-A$.}
\begin{proof}
Suppose the claim is false.  By symmetry, we can assume that $i=1$.  Apply Theorem~\ref{thm-sep} in $G$ for $x$
and $X_2$, obtaining sets $Y\subseteq X_2$ and
$K_0\subseteq E(G)$, where $k|Y|+|K_0|<s$, such that the component of
$(G-Y)-K_0$ that contains $x$ does not contain any vertex of $X_2$.  For each
$y\in Y$, apply Theorem~\ref{thm-sep} for $y$ and $X_1$, obtaining sets
$Y_y\subseteq X_1$ and $K_y\subseteq E(G)$, where $k|Y_y|+|K_y|<s$, such
that the component of $(G-Y_y) - K_y$ that contains $y$ does not contain any
vertex of $X_1$.  Let $K=K_0\cup\bigcup_{y\in Y} K_y$ and let
$Z=\bigcup_{y\in Y} Y_y$, and note that $|K|\le s^2$ and $|Z|\le s^2$.

By Menger's theorem, there exists a set $S_0$ of $w_0$ pairwise edge-disjoint paths from $x$ to $X_2$ in $G-A$.
Let $S\subseteq S_0$ consist of the paths that do not contain edges of $K$;
we have $|S|\ge w_0-s^2$.  Consider a path $P\in S$.  Let $v_0$, $v_1$,
\ldots, $v_t$ be the vertices of $P$ in order, where $v_0=x$ and $v_t\in X_2$.
As the component of $(G-Y)-K_0$ that contains $x$ does not contain any vertex of
$X_2$, the vertex $v_t$ belongs to $Y$.  Let $j$ be the largest index such that
$v_j$ belongs to $X_1$.  As the component of $(G-Y_{v_t}) - K_{v_t}$
that contains $v_t$ does not contain any vertex of $X_1$, it follows that $v_j$
belongs to $Y_{v_t}\subseteq Z$.  Consequently, $G-A$ contains a set of
$|S|$ pairwise edge-disjoint paths joining vertices of $Y$ with vertices of
$Z$ and otherwise disjoint from $W$.  By the pigeonhole principle, there
exist vertices $u\in Z\setminus A\subseteq X_1$ and $v\in Y\setminus A\subseteq X_2$ contained in at least
$\frac{|S|}{|Y||Z|}\ge m$ of these paths, and thus $uv$ is an edge of $G(m,W)$.
This contradicts the assumption that $(X_1, X_2)$ is a sepearation of $G(m,W)-A$.
\end{proof}

Consider any vertex $v\in X_1$.  Since $W$ is $w$-edge-connected, $G$ contains at least $w$ pairwise edge-disjoint
paths from $v$ to $X_2$.  By \refclaim{cl-single}, at least $w-w_0+1$ of these paths pass through a vertex of $A$.
By pigeonhole principle, we have the following.
\claim{cl-connecta}{For every $v\in X_1$, there exists a vertex $a_v\in A$ such that $G-(X_2\cup (A\setminus \{a_v\}))$
contains at least $s$ pairwise edge-disjoint paths from $v$ to $a_v$.}

Suppose that the lemma is false and that $G-A$ does not contain an edge-cut of size less than $w$ separating $X_1$ from $X_2$.
By Menger's theorem, $G$ contains a set $S_1$ of $w$ pairwise edge-disjoint paths from $X_1$ to $X_2$ and otherwise
disjoint from $W$.  By \refclaim{cl-single}, each vertex of $X_2$ is incident with less than $w_0$ of these paths,
and thus we can select $S_2\subseteq S_1$ such that $|S_2|\ge |S_1|/w_0\ge 2a_0s$ and the paths in $S_2$ have pairwise distinct
ends in $X_2$.  Furthermore, we can select $z\in A$ and $S_3\subseteq S_2$ of size at least $|S_2|/a_0\ge 2s$ such that
every $v\in X_1$ incident with a path in $S_3$ satisfies $a_v=z$.

Let $U$ be the set of endpoints of the paths of $S_3$ in $X_2$ and note that $|U|=|S_3|\ge 2s$.
Consider any sets $U'\subseteq U$ and $K\subseteq E(G)$ such that $k|U'|+|K|<s$.
We have $|U'|<s$, and thus $|U\setminus U'|\ge s$.  Since $|K|<s$, there exists a path $P\in S_3$
ending in $U\setminus U'$ and disjoint with $K$.  Let $v$ be the endpoint of $P$ in $X_1$.  By \refclaim{cl-connecta},
there exists a path from $v$ to $z$ disjoint with $K$ and $U'$.  Therefore, $(G-U')-K$ contains a path
from $z$ to $U$.  Since this holds for every $U'\subseteq U$ and $K\subseteq E(G)$ with $k|U'|+|K|<s$,
we obtain a contradiction with Theorem~\ref{thm-sep}.
\end{proof}

Furthermore, avoiding a strong immersion of a fixed graph restricts the structure of $G(m,W)$, as long as
$m$ is large enough.

\begin{lemma}\label{lemma-nok1k}
Let $F$ and $G$ be graphs, let $W$ be a subset of $V(G)$ and let $m$ be an integer.  If $m\ge 2|E(F)|$
and $G$ does not contain $F$ as a strong immersion, then $G(m,W)$ does not contain $K_{1,|V(F)|}$ as a minor.
\end{lemma}
\begin{proof}
The claim is trivial if $|V(F)|\le 1$.  Suppose that $|V(F)|\ge 2$ and that $G(m,W)$ contains $K_{1,|V(F)|}$ as a minor,
that is, $G(m,W)$ contains a subtree $T$ with $|V(F)|$ leaves and at least one non-leaf vertex.  Let $c$ be a non-leaf vertex of $T$ and let $Z$ be the
set of leaves of $T$.  Let $\theta$ be an injective function mapping $V(F)$ to $Z$.  By the definition of $G(m, W)$ and
Menger's theorem, there exists a set $S$ of $2|E(F)|$ pairwise edge-disjoint paths in $G$ from $Z$ to $c$, such that
every vertex $z\in Z$ is contained in exactly $\deg_F(\theta^{-1}(z))$ of these paths.  A \emph{half-edge} of $F$ is
a pair $(u,e)$, where $e$ is an edge of $F$ and $u$ is incident with $e$.  Note that there exists a bijective function $f$
from half-edges of $F$ to $S$ such that the path $f(u,e)$ contains the vertex $\theta(u)$ for every half-edge $(u,e)$.
We extend $\theta$ to a strong immersion of $F$ in $G$ by defining $\theta(e)=f(u,e)+f(v,e)$ for every edge $e=uv$ of $G$.
\end{proof}

The next lemma gives an approximate characterization of when a graph does not contain a large $K_{1,k}$ minor.   A set $X$ of vertices of a graph $G$ is a \emph{linearizing set} if $G-X$ is a vertex-disjoint union of paths.

\begin{lemma}[\cite{MW}]\label{lemma-deg3}
If a simple connected graph $G$ does not contain $K_{1,k}$ as a minor, then $G$ has a linearizing vertex set of size at most $4k$.
\end{lemma}

We now give the proof of Theorem \ref{thm:mainstrong}.
\begin{proof}[Proof of Theorem~\ref{thm:mainstrong}]
Let $m=2|E(F)|$, $a=4|V(F)|$, $a_0=a+1$ and let $w$ be the corresponding constant from Lemma~\ref{lemma-vcon}.
Let $k=\max(\Delta(F),3)$ and $p=3d(k)|V(F)|+1$.

Let $W$ be a $w$-edge-connected subset of $V(G)$, and let $H=G(m,W)$.  If $H$ is not connected, then there exists a separation $(X_1, X_2)$ of $H$.  By Lemma~\ref{lemma-vcon} applied with $A=\emptyset$,
it follows that $G$ contains an edge-cut of size less than $w$ separating $X_1$ from $X_2$.  This is a contradiction, since $W$ is
$w$-edge-connected.

Therefore, $H$ is a connected simple graph, and by Lemma~\ref{lemma-nok1k}, $H$ does not contain $K_{1,|V(F)|}$ as a minor.
Let $A$ be the smallest linearizing set in $H$.  By Lemma~\ref{lemma-deg3}, we have $|A|\le a$.
Let $x_1$, $x_2$, \ldots, $x_t$ be an ordering of $W\setminus A$ such that for $2\le i\le t-1$, the neighbors of $x_i$ in $H-A$ are contained
in $\{x_{i-1}, x_{i+1}\}$.

If $t=0$, we set $B_0=V(G)\setminus A$.  If $t=1$, we set $B_0=\emptyset$ and $B_1=V(G)\setminus (A\cup \{x_1\})$.  If $t=2$, we set
$B_0=B_2=\emptyset$ and $B_1=V(G)\setminus (A\cup \{x_1,x_2\})$.  In all the cases, we obtain a path-like decomposition of $G-A$
with respect to $W\setminus A$ of width $0$, and since $p\ge 3$, the neigborhood of every vertex of $A$ is $p$-bounded.  Therefore, assume that $t\ge 3$.

Consider an index $i$ such that $2\le i\le t-1$ and let $U_i=\{x_1,\ldots, x_{i-1}\}$ and $V_i=\{x_{i+1},\ldots, x_t\}$.
A set $Z\subset V(G)\setminus A$ is an \emph{$i$-separator} if $x_i\not\in Z$, $U_i\subseteq Z$ and $V_i\cap Z=\emptyset$.
We set $s_i(Z)$ to be the number of edges of $G-A$ with one end in $Z$ and the other end in $V(G-A)\setminus (Z\cup \{x_i\})$.
Let $L_i\subset V(G)\setminus A$ be an $i$-separator with $s_i(L_i)$ as small as possible, and subject to that with $|L_i|$ minimal.
By Lemma~\ref{lemma-vcon} applied with $A\cup \{x_i\}$, we have $s_i(L_i)<w$.

\claim{cl-subset}{If $2\le i<j\le t-1$, then $L_i\subset L_j$.}
\begin{proof}
Consider indices $i$ and $j$ such that $2\le i<j\le t-1$.  Let $N_i=L_i\cap L_j$ and $N_j=L_i\cup L_j$.  Note that $N_i$ is an $i$-separator
and $N_j$ is a $j$-separator, and thus $s_i(N_i)\ge s_i(L_i)$ and $s_j(N_j)\ge s_j(L_j)$.  Observe that
$s_i(L_i)+s_j(L_j)-(s_i(N_i)+s_j(N_j))=2b_1+b_2$, where $b_1$ is the number of edges with one end in $L_j\setminus (L_i\cup \{x_i\})$ and the other
end in $L_i\setminus L_j$, and $b_2$ is the number of edges incident with $x_i$ or $x_j$ and with the other end in $L_i\setminus L_j$.
Consequently, $s_i(L_i)+s_j(L_j)\ge s_i(N_i)+s_j(N_j)$.  Putting the inequalities together, we have
$s_i(N_i)=s_i(L_i)$ and $s_j(N_j)=s_j(L_j)$.  Since $L_i$ is chosen with $|L_i|$ minimal, it follows that $|N_i|\ge |L_i|$.
Therefore, $L_i\subseteq L_j$.  Furthermore, the inclusion is sharp, since
$x_i\in L_j\setminus L_i$.
\end{proof}

Let $L_1=\emptyset$ and $L_t=V(G)\setminus (A\cup \{x_t\})$.
Let $B_0=B_t=\emptyset$, and for $1\le i\le t-1$, let us set $B_i=L_{i+1}\setminus (L_i\cup\{x_i\})$.
By \refclaim{cl-subset}, $B_0$, \ldots, $B_t$ is a near-partition of $V(G)\setminus W$, and thus the ordering $x_1$, \ldots, $x_t$
and the sets $B_0$, \ldots, $B_t$ form a path-like decomposition $\PP$ of $G-A$ with respect to $W\setminus A$.
Since $s_i(L_i)<w$ for $2\le i\le t-1$, the width of $\PP$ is less than $w$.

\claim{cl-comp}{For $1\le i\le t-2$, each component of $G[B_i]$ contains a neighbor
of $x_i$ or $x_{i+1}$.}
\begin{proof}
Suppose that $C$ is the vertex set of a component of $G[B_i]$ containing neighbors of neither $x_i$ nor $x_{i+1}$.
We say that an edge of $G-A$ with one end in $C$ and the other end not in $C$ is \emph{backward} if its end not in $C$
belongs to $L_{i+1}$, and it is \emph{forward} otherwise.  Note than neither forward nor backward edges are incident with $x_{i+1}$.

Let $L'_{i+1}=L_{i+1}\setminus C$ and note that $L'_{i+1}$ is an $(i+1)$-separator.
Since $|L'_{i+1}|<|L_{i+1}|$, the choice of $L'_{i+1}$ implies that $s_{i+1}(L'_{i+1})>s_{i+1}(L_{i+1})$, and thus
there are more backward edges than forward ones.  Since $B_i=L_{i+1}\setminus (L_i\cup\{x_i\})$ and $C$ induces a component
of $G[B_i]$ containing no neighbors of $x_i$, all backward edges are incident with vertices in $L_i$.
Since $L_1=\emptyset$, we have $i\ge 2$.  However, then $L'_i=L_i\cup C$ is an $i$-separator with $s_i(L'_i)<s_i(L_i)$,
which is a contradiction.
\end{proof}

To complete the proof, we must show that every vertex of $A$ is $p$-bounded in $\PP$.  Suppose that the neighborhood of a vertex $c\in A$ is not $p$-bounded.
By \refclaim{cl-comp}, $G$ contains at least $(p-1)/3=d(k)|V(F)|$ paths from $c$ to vertices of $W\setminus A$ whose vertex sets pairwise intersect
only in $c$.  Let $X$ denote the set of their endpoints in $W\setminus A$.  Suppose that sets $Y\subseteq X$ and
$K\subseteq E(G)$ have the property that the component of $(G-Y)-K$ that contains $c$ does not contain any vertex of $X$.
Then each of the paths from $c$ to $X$ contains either a vertex of $Y$ or an edge of $K$, and thus
$k|Y|+|K|\ge d(k)|V(F)|$.  This contradicts Theorem~\ref{thm-sep}.
\end{proof}

\section{An approximate converse}\label{sec:converse}

We conclude by showing that the decomposition guaranteed by Theorem \ref{thm:mainstrong} does indeed give a good approximation of graphs excluding strong clique immersions.  

\begin{theorem}\label{thm-converse}
For all integers $d$, $a$, $w$ and $p$, there exists an integer $n$ such that
if every $d$-edge-connected subset of $V(G)$ is $(a,w,p)$-linear, then $G$ does not contain $K_n$ as a strong immersion.
\end{theorem}
\begin{proof}
Let $n=\max\{2w+2,(2p+1)a+1,d+1\}$, and suppose that $\theta$ is a strong immersion of $K_n$ in $G$.  Let $V$ be the vertex set of $K_n$ and let $W=\theta(V)$.
Since $n-1\ge d$, $W$ is $d$-edge-connected in $G$, and thus it is $(a,w,p)$-linear.  Let $A$ be a subset of $W$ of size at most $a$
such that $G-A$ has a path-like decomposition $\PP$ with respect to $W\setminus A$ of width less than $w$
and the neigborhood of every vertex of $A$ is $p$-bounded in $\PP$.  Let $x_1$, \ldots, $x_m$ be the ordering of $W\setminus A$
according to $\PP$ and let $B_0$, \ldots, $B_m$ be the bags of $\PP$.  Since $m\ge n-a\ge 2|A|p+1$ and the neighborhood of every
vertex of $A$ in $G-A$ is $p$-bounded in $\PP$, there exists an index $1\le i\le m$ such that the vertices of $A$ have no neighbors
in $\{x_i\}\cup B_{i-1}\cup B_i$.  Let $K$ be the union of the $x_{i-1}$-cut and the $x_{i+1}$-cut of $\PP$.
Consider an edge $e$ of $K_n$ incident with $\theta^{-1}(x_i)$ such that $e$ is incident neither
with $\theta^{-1}(x_{i-1})$ nor $\theta^{-1}(x_{i+1})$.  Since the immersion is strong, $\theta(e)$ contains neither $x_{i-1}$ nor $x_{i+1}$,
and thus $\theta(e)$ contains an edge of $K$.  Therefore, there are at most $|K|\le 2(w-1)$ such edges incident with $x_i$, which is a contradiction
since $n\ge 2w+2$.
\end{proof}

Similarly, the structure described in Theorem~\ref{thm:weakdecomp2} is sufficient to exclude a large clique as a strong immersion.

\begin{theorem}\label{thm-convstr}
For every integer $\alpha\ge 1$ there exists an integer $n\ge 1$ such that 
if a graph $G$ has a tree-cut decomposition $(T, \zX)$ of adhesion
less than $\alpha$ such that each torso is $\alpha$-basic,
then $G$ does not contain $K_n$ as a strong immersion.
\end{theorem}
\begin{proof}
Let $n=2\alpha^2+2\alpha+1$, and suppose that $\theta$ is a strong immersion of $K_n$ in $G$.  Let $V$ be the vertex set of $K_n$ and let $W=\theta(V)$.
The set $W$ is $\alpha$-edge-connected in $G$, and since $(T, \zX)$ has adhesion less than $\alpha$, we conclude that there exists $t\in V(T)$
such that $W\subseteq X_t$.  Let $H$ be the torso of $(G, T, \zX)$ at $t$, and observe that $H$ also contains $K_n$ as a strong immersion.
However, we can now obtain a contradiction in the same way as in the proof of Theorem~\ref{thm-converse}.
\end{proof}

\bibliographystyle{siam}
\bibliography{strongimm}

\end{document}